\documentclass[11pt]{amsart}
\usepackage[margin=2.5cm]{geometry} 

\usepackage{mathrsfs, amsmath, amsthm, amsfonts, amssymb, amscd, graphicx, epstopdf, color}
\usepackage{pinlabel}

\usepackage{setspace}
\usepackage{hyperref}

\newtheorem{thm}{Theorem}[section]
\newtheorem{lem}[thm]{Lemma}
\newtheorem{prop}[thm]{Proposition}
\newtheorem{cor}[thm]{Corollary}
\newtheorem{claim}[thm]{Claim}

\newtheorem*{thm*}{Main Theorem}
    
\theoremstyle{definition}

\newtheorem{remark}[thm]{Remark}

\def\Z{\mathbb{Z}}
\def\S{\mathbb{S}}

\def\Q{\mathbb{Q}}

\def\bs{\boldsymbol}

\def\a{\alpha}
\def\b{\beta}

\def\g{\gamma}

\def\vempty{\varnothing}
\def\d{\partial}

\def\<{\langle}
\def\>{\rangle}

\def\cl{\overline}

\title{On nonsimple knots in lens spaces with tunnel number one}

\author{Michael J. Williams}
\address{\hskip-\parindent
	Mathematics Department\\
	University of California\\
	Santa Barbara, CA 93106\\
	USA}
\email{mikew@math.ucsb.edu}
\thanks{Research supported by NSF VIGRE Grant No. DMS-0135345 and a UCOP Postdoctoral Fellowship.}
\date{\today} 

\begin{document}

\begin{abstract} 
A knot $k$ in a closed orientable 3-manifold is called nonsimple if the exterior of $k$ possesses a properly embedded essential surface of nonnegative Euler characteristic. We show that if $k$ is a nonsimple prime tunnel number one knot in a lens space $M$ (where $M$ does not contain any embedded Klein bottles), then $k$ is a $(1,1)$ knot. Elements of the proof include handle addition and Dehn filling results/techniques of Jaco, Eudave-Mu\~noz and Gordon as well as structure results of Schultens on the Heegaard splittings of graph manifolds. 
\end{abstract}	

\maketitle

\section{Introduction} \label{sec:introduction}
This paper is devoted to establishing the following.

\begin{thm*} 
Let $k$ be a nonsimple prime tunnel number one knot in a lens space $M$, where $M$ contains no embedded Klein bottles. Then $k$ is a $(1,1)$ knot.
\end{thm*}

This result is known in the case of $M \cong \S^3$ by \cite{ms:unknotting} and \cite{eudave:nonsimple} independently; the assumption that $k$ is prime is not needed for these results. It was proved in \cite{norwood:every} (see also \cite{scharlemann:tunnel}) that all tunnel number one knots in $\S^3$ are prime. The papers \cite{doll:generalized} and \cite{hayashi:satellite} contain results about nonsimple $(1,1)$ knots in lens spaces; these papers also contain results about composite $(1,1)$ knots. 

There are tunnel number one knots in lens spaces that are not $(1,1)$ knots by \cite{msy:examples}, \cite{mms:high}, \cite{mr:heegaard}, \cite{eudave:incompressible}, \cite{eudave:incompressible2}, \cite{jt:tunnel}. The main theorem of this paper then implies that if the lens space has no embedded Klein bottles, then such knots must be hyperbolic or composite. 

By \cite{bw:nonorientable}, the condition that $M$ contains no embedded Klein bottles is equivalent to \mbox{$M \ncong L(4n, 2n-1)$} for any $n \in \Z$. Nevertheless, some results will be established where this restriction on $M$ can be relaxed.

Most of the content of this article stems from my graduate work at UC Davis. I am grateful for all of the support, encouragement, and insight that I received from my PhD advisor Abigail Thompson. I would like to thank Martin Scharlemann for many helpful suggestions for this paper and for his general support. 

\section{Background} \label{sec:background}
Let $X$ be a closed, connected, orientable $3$--manifold. If $Y$ is a finite subcomplex of $X$, let $N(Y)$ denote a regular neighborhood of $Y$ in $X$. In particular, if $k$ is a knot in $X$, then $N(k)$ is a solid torus neighborhood of $k$, and we define $E(k)=\cl{X - N(k)}$ to be the \textbf{exterior} of $k$ in $X$. A knot $k$ in $X$ has \textbf{tunnel number one} if there is a properly embedded arc $\g$ in $E(k)$ such that $\cl{E(k) - N(\g)}$ is a genus $2$ handlebody; equivalently, $E(k)$ admits a genus $2$ Heegaard splitting. In this case, we call the arc $\g$ an \textbf{unknotting tunnel} for $k$.

A \textbf{lens space} is a closed, orientable 3-manifold $M$ with finite (cyclic) fundamental group and Heegaard genus at most one; this includes $\S^3$, but not $\S^1 \times \S^2$. A genus 1 Heegaard splitting for a lens space is equivalent to an $r$--Dehn surgery on the unknot in $\S^3$ where $r \ne0$. Let $L(p,q)$ denote $\frac{p}{q}$--Dehn surgery on the unknot in $\S^3$ with $\gcd(p,q)=1$ and $p \ne 0$, as in \cite{rolfsen:knots90}. Note that $\S^3 \cong L(1,q)$ for any $q \in \Z$. 

For a lens space $M$, we need to make the following distinction between a \textit{trivial knot} in $M$ and an \textit{unknotted knot} in $M$. A knot $k$ in $M$ is \textbf{trivial} if it bounds an embedded disk; note that $E(k) \cong M \# (\S^1 \times D^2)$. A knot $k$ in $M$ is \textbf{unknotted} if $E(k)$ is a solid torus. It follows from the Loop Theorem that the notions of trivial knot and unknotted knot coincide if and only if $M \cong \S^3$. 

We say that a knot $k$ in a lens space $M$ is \textbf{nonsimple} if $E(k)$ is a nonsimple $3$--manifold, that is, $E(k)$ contains an   essential sphere, disk, annulus, or torus. Otherwise, $k$ is called \textbf{simple}; then it follows, by work of Thurston, that $k$ is a \textit{hyperbolic knot}, i.e. $M-k$ admits a hyperbolic structure. When $M$ is $\S^3$, a knot is nonsimple if it is a trivial knot, a torus knot, or a satellite knot. When $M$ is not $\S^3$, we have a somewhat similar statement, once we give some definitions.

We say that a knot $k$ in a lens space $M$ is a \textbf{torus knot} if $k$ can be isotoped to lie as an essential curve on a Heegaard torus for $M$. In this case, $E(k)$ admits a Seifert fibration over the disk with at most two exceptional fibers. We consider an unknotted knot to be a torus knot.

We say that a knot $k$ in a lens space $M$ is a \textbf{composite knot} if $E(k)$ contains an essential meridional annulus; otherwise, we say that $k$ is a \textbf{prime knot}. If $k$ is composite, then an essential meridional annulus $A$ in $E(k)$ completes to a  2-sphere $S$ in $M$ that intersects $k$ in $2$ points. Furthermore, $S$ bounds a 3-ball $B \subset M$, so $S$ decomposes the pair $(M,k)$ into $(M-B,k_1) \cup (B,k_2)$ where each $k_i$ is an arc, and $k=k_1 \cup k_2$. 

We say that a knot $k$ in a lens space $M$ is a \textbf{satellite knot} if $E(k)$ has an essential torus. It is well-known that if $k$ is a composite knot in $M$, then an essential ``swallow follow" torus $T \subset E(k)$ can be formed by appropriately choosing an annulus $A' \subset \d E(k)$ cut-off by $A$, and pushing the torus $A \cup A'$ into the interior of $E(k)$. 

We say that a knot $k$ in a lens space $M$ is \textbf{cabled on $\bs{k_0}$} if $k_0$ is knot in $M$ such that $k$ lies as an essential curve on $\d N(k_0)$, and $k$ is neither a longitude nor a meridian of $k_0$. If, in addition, $E(k_0)$ has incompressible boundary, then $k$ is also a satellite knot; this will be the case that we will be most interested in.

We say that a knot $k$ in a lens space $M$ \textbf{admits a $\bs{(1,1)}$ decomposition} if there is a genus 1 Heegaard splitting $M=V_1 \cup_T V_2$ so that $k \cap V_i$ is a properly embedded trivial arc in $V_i$ for each $i=1,2$. A knot $k$ in $M$ that admits a $(1,1)$ decomposition is called a \textbf{$\bs{(1,1)}$ knot}. It is not difficult to prove that $(1,1)$ knots have tunnel number one.

The following is a generalization of \cite[Lemma 4.4]{eudave:nonsimple}.

\begin{lem} \label{lem:1-bridge}
Let $k$ be a tunnel number one knot in a lens space $M$. Suppose that $k$ has an unknotting tunnel $\tau$ that can be slid over itself to obtain $\tau_1 \cup \tau_2$ where $\tau_1$ is an unknotted simple closed curve  and $\tau_2$ is an arc connecting $\tau_1$ to $k$. Then $k$ is a $(1,1)$ knot.
\end{lem}

\begin{proof}
Let $\g$ denote the graph $\tau_1 \cup \tau_2$.
Note that $N(\g)$ is isotopic to $N(\tau_1)$, so $E(\g)$ is a solid torus. 
Consider the genus 1 Heegaard splitting $M=N(\g) \cup E(\g)$. 
Set $k_N=k \cap N(\g)$ and $k_E=k \cap E(\g)$. 
It is clear that $k_N$ is a trivial arc in $N(\g)$. 
Note that $k_E$ is an unknotting tunnel for $E(\g)$. 
By \cite[Theorem $1'$]{gordon:primitive}, $k_E$ is unknotted in $E(\g)$. 
This gives a $(1,1)$ decomposition for $k$. 
\end{proof}

For background in knot theory and Dehn surgery, see \cite{rolfsen:knots90}, \cite{boyer:dehn} and \cite{gordon:dehn_filling}. For background in the theory of Heegaard splittings, see \cite{scharlemann:heegaard}. If $\a$ and $\b$ are slopes (isotopy classes of simple closed curves) on a torus $T$, let $\Delta(\a,\b)$ denote their minimal geometric intersection number on $T$; by abuse of notation, we may replace $\a$ and/or $\b$ by a simple closed curve representative. Once we identify the slopes of $T$ with $\Q \cup \{1/0\}$, we can calculate intersection numbers via the formula $\Delta(a/b,c/d)=|ad-bc|$.

\section{Irreducibility} \label{sec:irred}
A famous theorem of J. Alexander implies that for every knot $k \subset \S^3$, the exterior $E(k)$ is irreducible. One of the main results in this section is that when restricted to tunnel number one knots in lens spaces, only the trivial knot has reducible exterior. Before proving this, we need the following lemma.

\begin{lem} \label{lem:boundary_reducible}
If $k$ is a knot in a closed, connected, orientable $3$-manifold  $M$ and $E(k)$ has compressible boundary, then $E(k) \cong \S^1 \times D^2$ or $E(k) \cong M \# (\S^1 \times D^2)$.
\end{lem}

\begin{proof}
First, we compress the boundary of $E(k)$ via $D$ to get a 2-sphere $S$ that is disjoint from $\d E(k)$. Now $S$ cuts $M$ into $\cl{M - B}$ and $B$ where $B$ is a 3-ball. To recover $E(k)$, we reverse the compression by $D$. This is equivalent to attaching a 1-handle to $B$ or $\cl{M - B}$. Since we are attaching a 1-handle to a 3-manifold with 2-sphere boundary, we may conclude that $E(k) = \S^1 \times D^2$ or $E(k) = M \# (\S^1 \times D^2)$. 
\end{proof}

\begin{prop} \label{prop:haken} 
Let $k$ be a nontrivial tunnel number one knot in a lens space  $M$. Then $E(k)$ is irreducible. Consequently, $E(k)$ is a Haken 3-manifold.
\end{prop}

\begin{proof}
Suppose that $E(k)$ is reducible. We will show that this implies that $k$ is the trivial knot, thereby contradicting our hypothesis on $k$. 
Since $E(k)$ is reducible, we must have $M \ncong \S^3$.
Let $N(k)$ be a regular neighborhood of $k$ in $M$. 
Since $k$ has tunnel number one, $E(k)$ can be realized as $H_\a$, the 3-manifold obtained by attaching a 2-handle to a genus 2 handlebody $H$ along a nonseparating simple closed curve $\a \subset \d H$. Note that $\a$ does not bound a disk in $H$, since $M \ncong \S^1 \times \S^2$.
By Jaco's 2-handle addition \cite{jaco:adding}, $\d H - \a$ is compressible in $H$. 
So there is an essential disk $D_1 \subset H$ that misses $\a$. 

We can take $D_1$ to be nonseparating in $\d H - \a$. Otherwise $D_1$ cuts off from $\d H - \a$ a once or twice punctured disk $P$. If $P$ has one puncture, then $\d D_1$ would be parallel to $\a$; this contradicts that $\a$ does not bound a disk in $H$. If $P$ has two punctures, then $\d D_1$ would be separating on $\d H$. Hence $D_1$ would separate $H$ into two solid tori. Only one of these solid tori meets the 2-handle, so we can redefine $D_1$ to be a meridian disk of the other solid torus. Now $D_1$ is nonseparating in $H$. Moreover, $\d D_1$ is nonseparating in $\d H - \a$. We now see that $D_1$ persists as a nonseparating (hence essential) disk of $\d E(k)$. By Lemma~\ref{lem:boundary_reducible}, we conclude that $E(k) \cong M \#(\S^1 \times D^2)$. 

Let $D_2 \subset H$ be a meridian disk such that $D_2$ is disjoint from $D_1$ and $D_1\cup D_2$ cuts $H$ into a $3$--ball. Since $\a$ is a nonseparating curve on $\d H$ and $\a$ is disjoint from $D_1$, $\a$ must intersect $D_2$. 
Now $M=E(k) \cup_\d N(k) = [M \# (\S^1 \times D^2)] \cup_\d N(k)$. 
Let $D$ be a meridian disk of $N(k)$. 
So, on the torus $\d E(k)=\d N(k)$, we must have $\Delta(\d D_2, \d D)=1$.
Thus $\d D_2$ is parallel to $k$ within $N(k)$.
It follows that $k$ is a trivial knot, contradicting our hypothesis. 
\end{proof}

\begin{prop} \label{prop:knots}
Let $k$ be a knot in a lens space $M$ such that $E(k)$ is irreducible. Then $k$ is either a torus knot, a satellite knot, or a hyperbolic knot. 
\end{prop}

\begin{proof}
By Thurston's Hyperbolization Theorem for Haken 3-manifolds, $E(k)$ is Seifert, toroidal or hyperbolic. Suppose that $k$ is neither a satellite knot nor a hyperbolic knot. Then $E(k)$ is an atoroidal Seifert manifold. We show that $k$ is a torus knot. 

Fix some Seifert fibration for $E(k)$. Let $F$ denote the base orbifold. By \cite{heil:elementary}, the Seifert fibration for $E(k)$ extends to one on $M=E(k) \cup N(k)$ in the natural way since $M$ is irreducible. Let $\hat{F}$ be the base orbifold of this extended Seifert fibration on $M$. 

By \cite[Satz 9]{seifert:topologie}, $\hat{F}$ must be either

\begin{itemize}
\item the $2$--sphere with at most two cone points, or
\item the projective plane with at most one cone point.
\end{itemize}

Hence $F$ must be either 

\begin{itemize}
\item the disk with at most two cone points, or
\item the M\"obius band with at most one cone point.
\end{itemize}

If $F$ is a disk with at most two cone points, then $k$ is a torus knot. If the underlying surface of $F$ is a M\"obius band, then $F$ must have no cone points, else $E(k)$ would contain an essential torus. If $F$ is a M\"obius band, then $E(k)$ is a circle bundle over the M\"obius band, and hence homeomorphic to the twisted $I$--bundle over the Klein bottle. Thus $E(k)$ admits another Seifert fibration with base orbifold a disk with two cone points; see \cite[VI.5.(d)]{jaco:lectures}. Thus $k$ is a torus knot. 
\end{proof}

\begin{remark} \label{remark:seifert_toroidal}
$E(k)$ could be simultaneously Seifert and toroidal if $E(k)$ admits a Seifert fibration over the M\"obius band with one exceptional fiber. Consequently, $E(k)$ will contain an embedded Klein bottle.
\end{remark} 

\begin{cor} \label{cor:haken}
Let $k$ be a nonsimple tunnel number one knot in a lens space $M$. Then $k$ is either a trivial knot, a torus knot, or a satellite knot.
\end{cor}

\section{Satellite Knots} \label{sec:satellites}
We show that the Main Theorem holds under the assumption that $k$ is a satellite knot, that is, $E(k)$ contains an essential torus. These two cases are dealt with in Propositions~\ref{prop:anannular}~and~\ref{prop:annular}. 

\begin{prop} \label{prop:anannular}
Let $k$ be a tunnel number one satellite knot in a lens space $M$. Suppose that $E(k)$ contains no essential annulus. Then $k$ is a $(1,1)$ knot. 
\end{prop}

\begin{proof}
Under this hypothesis, the knot $k$ is nontrivial and knotted. Then our proof of Proposition~\ref{prop:haken} shows that \cite[Theorem 2 (c)]{eudave:nonsimple} applies and provides the following description: Let $\tau$ be an unknotting tunnel for $E(k)$. $E(k)$ has an essential torus $T$ such that the tunnel $\tau$ can be slid over itself to obtain $\tau_1 \cup \tau_2$ where $\tau_1$ is a simple closed curve, $\tau_2$ is an arc, and $T$ meets $\tau$ in a point in $\tau_2$. Futhermore, $E(k)$ can be decomposed into $E(k)=Y' \cup_T Y''$, where $Y'$ is obtained by identifying two solid tori $V_1$, $V_2$ along an annulus $A$, $\tau_1$ is a core of $V_1$, $V_1 \cap \tau_2$ is a straight arc joining $\d V_1 - A$ and $\tau_1$, and $V_2 \cap \tau = \vempty$. 

When we form the lens space $M=E(k) \cup N(k)=Y' \cup_T (Y'' \cup N(k))$, the torus $T$ will be compressible in $V=Y'' \cup N(k)$ in $M$; we show that $V \cong \S^1 \times D^2$. Let $D$ be a compressing disk for $T$. Then $T$ compresses to a 2-sphere $S$. Then $S$ is separating. Since $M$ is prime, $S$ must bound a 3-ball on one side. The side of $S$ containing the disk $D$ is the manifold obtained by attaching a 2-handle (a collar of $D$) to $Y'$. This 3-manifold is homeomorphic to either a punctured lens space (if $\d D$ is not a fiber of the Seifert-fibered space $Y'$) or the connected sum of two lens spaces (if $\d D$ is a fiber of the Seifert-fibered space $Y'$); the latter scenario is impossible since $M$ is prime. Thus $S$ bounds a 3-ball disjoint from $D$. Therefore, $V \cong \S^1 \times D^2$. 

Since $V \cong \S^1 \times D^2$, $M$ is obtained as a Dehn filling on the Seifert-fibered space $Y'$. We extend the Seifert fibration of $Y'$ to one on $M$. Note that $Y'$ has a Seifert fibration over the disk with at most two exceptional fibers. We now see that $\d V_1$ is a Heegaard torus for $M$, hence $\tau_1$ is an unknotted simple closed curve in $M$. It follows from Lemma~\ref{lem:1-bridge} that $k$ is a $(1,1)$ knot.
\end{proof}

\begin{prop} \label{prop:annular}
Let $k$ be a tunnel number one knot in a lens space $M$, where $M$ contains no embedded Klein bottles. Suppose that $E(k)$ contains an essential torus and an essential annulus. If $k$ is prime, then $k$ is a $(1,1)$ knot.
\end{prop}

Under this hypothesis, we will show that $E(k)$ is \textit{cabled on a torus knot exterior}, then use this structure of $E(k)$ and the induced genus 2 Heegaard splitting coming from a tunnel to obtain a $(1,1)$ decomposition for $E(k)$. 

Let $C_{p,q}$ $(q >1)$ denote the space obtained by removing a regular fiber from a solid torus $V$ with a $(p,q)$--fibering. We call $C_{p,q}$ a \textbf{cable space of type $\bs{(p,q)}$}. Let $T$ and $T'$ denote the boundary tori of $C_{p,q}$. There is a standard way to embed $C_{p,q}$ in $\S^3$ so that a regular fiber of is a $(p,q)$--torus knot and $T$ is a Heegaard torus. Now it is easy to see that there are slope co-ordinates for $T$ and $T'$ so that a regular fiber of the $(p,q)$--fibering on $C_{p,q}$ has slope $p/q$ on $T$, and slope $pq$ on $T'$; we call these slope co-ordinates the \textbf{standard slope co-ordinates}; see \cite[Section 7]{gordon:dehn_surgery}.

We will call a 3-manifold $X$ \textbf{cabled (on $\bs{Y}$)} if $X=Y \cup_T C_{p,q}$ where $T=Y \cap C_{p,q} \subset \d C_{p,q}$ is a (possibly compressible) torus in $Y$, and $T'$ is the other boundary component of $C_{p,q}$; we will typically have $\d X =T'$. 

\begin{remark} \label{remark:cable}
In the case of knot exteriors, if $k$ is a knot in a closed orientable 3-manifold so that $k$ is cabled on a knot $k_0$, then $E(k)$ is cabled on $E(k_0)$. So $E(k)=E(k_0) \cup_{T} C_{p,q}$ for some $q \ge 2$ with $T= \d E(k_0)$. To get the standard slope co-ordinates on $\d C_{p,q}$, we let $\mu,\lambda$ be a meridian-longitude pair for $N(k_0)$ and let $\mu', \lambda'$ be an appropriate meridian-longitude pair (choice of $\lambda'$ depends on $\lambda$) for $N(k)$. Then a slope on $\d N(k)$ has co-ordinates $m/l$ where $m,l \in \Z$ and $\gcd(m,l)=1$ if some oriented simple closed curve representing the slope goes homologically $m$-times in the $\mu$ direction and $l$-times in the $\lambda$ direction. We similarly have slope co-ordinates on $\d N(k_0)$ with respect to $\mu'$ and $\lambda'$. With these choices of co-ordinates, we see that $k(pq)=k_0(p/q) \# L(q,p)$; see \cite[Lemma 7.2, Corollary 7.3]{gordon:dehn_surgery}.
\end{remark}

\begin{lem} \label{lem:cable}
Let $k$ be a tunnel number one satellite knot in a lens space $M$, where $M$ contains no embedded Klein bottles. Assume that $E(k)$ contains an essential nonmeridional annulus. Then $k$ is cabled. 
\end{lem}

\begin{proof} 
Assume the hypotheses. Also, we may as well assume that $E(k)$ is not a Seifert manifold by Remark~\ref{remark:seifert_toroidal}. Since $k$ has tunnel number one, $E(k)$ can be realized as $H_\a$, the 3-manifold obtained by attaching a 2-handle to a genus 2 handlebody $H$ along a nonseparating simple closed curve $\a \subset \d H$. We apply \cite[Theorem 1]{eudave:nonsimple}. Suppose that part (a) of this theorem holds. Then $H$ contains an $\a$--essential annulus $A$; furthermore $A$ will persist as an essential nonmeridional annulus in $H_\a=E(k)$ by \cite[Remark (d)]{eudave:nonsimple}; in particular, $A$ misses the 2-handle. Then $A$ cuts $\d E(k)$ into two annuli $A_1$ and $A_2$. Let $S=A \cup A_1$ in $E(k)$. 

Since $M$ contains no embedded Klein bottles, $S$ must be a separating torus. So $A$ separates $E(k)$ into 3-manifolds $N_1$ and $N_2$. Since $E(k)$ has torus boundary, $\d N_1$ and $\d N_2$ must be tori. Since $A$ is $\a$--essential, the 2-handle misses one of the $N_i$, say, $N_1$. Now $A$ separates $H$ into $N_1$ and $\cl{N_2 - N(\tau)}$ where $\tau$ is the co-core of the 2-handle attachment for $E(k)$. Since $A$ is essential in $H$, it must separate $H$ into handlebodies by \cite{kobayashi:structures}. Hence $N_1$ is a solid torus. Let $k_0$ denote the core of $N_1$. Since $A$ is essential in $E(k)$, $\d A$ are not longitudes of $k_0$; otherwise, $A$ would be parallel to $\d E(k)$ through $N_1 \subset E(k)$. It follows that $k$ is cabled on $k_0$.

Now suppose that part (b) of \cite[Theorem 1]{eudave:nonsimple} holds. Then there is an essential annulus $A$ of $E(k)$ that cuts $E(k)$ into $N_1$ and $N_2$, where $N_1$ is a solid torus. As in the end of the previous paragraph, it follows that $k$ is cabled on the core $k_0$ of $N_1$.
\end{proof}

\begin{lem} \label{lem:cable_torus} 
Let $k$ be a nontrivial tunnel number one cable knot in a lens space $M$, where $M$ contains no Klein bottles. Then $k$ is cabled on a torus knot.
\end{lem}

\begin{proof}
We merely mimic the proof of the case $M=\S^3$ from \cite[Lemma 4.6]{eudave:nonsimple}, but suitably generalized. Suppose that $k$ is cabled on a knot $k_0$. Then $E(k)=E(k_0) \cup_{T_0} C_{p,q}$ where $q \ge 2$. We use standard slope co-ordinates for $\d C_{p,q}$, as discussed in Remark~\ref{remark:cable}. Thus $k(pq)=k_0(p/q) \# L(q,p)$. By additivity of Heegaard genus under connected sums, we deduce that $k_0(p/q)$ is a lens space or $\S^1 \times \S^2$. Thus $k_0(1/0)$ and $k_0(p/q)$ both have cyclic fundamental groups with $\Delta(1/0,p/q)=q \ge 2$. By the Cyclic Surgery Theorem of \cite{cgls:dehn} and Proposition~\ref{prop:haken}, we conclude that $E(k_0)$ is a Seifert manifold. By Proposition~\ref{prop:knots} and Remark~\ref{remark:seifert_toroidal}, we have that $k_0$ is a torus knot.
\end{proof} 

\begin{proof}[Proof of Proposition~\ref{prop:annular}]
Since $E(k)$ is toroidal and contains no embedded Klein bottles, it cannot be a Seifert manifold; see Remark~\ref{remark:seifert_toroidal}. Let $\tau$ be an unknotting tunnel for $k$. Then $\tau$ induces a natural genus 2 Heegaard splitting for $E(k)$ which persists as a Heegaard surface for $M$ under the trivial Dehn filling of $N(k)$. Let this Heegaard splitting for $E(k)$ be denoted by $E(k)=V \cup_F W$ where $V$ is a compression body with $\d E(k)=\d_-V$, $W$ is a handlebody, and $F=\d_+V=\d_+W$. 

By Lemmas~\ref{lem:cable} and \ref{lem:cable_torus}, $E(k)$ can be expressed as $Y \cup_T C_{p,q}$ where $Y$ is a Seifert fibered space over $D^2$ with two exceptional fibers and $C_{p,q}$ is a cable space. To complete the proof of Proposition~\ref{prop:annular}, we need to analyze how the Heegaard surface $
F$ can intersect the Seifert pieces $Y$ and $C_{p,q}$ in $E(k)$. To this end, we will extensively use results of Schultens \cite{schultens:heegaard}, especially \cite[Theorem 1.1]{schultens:heegaard}  which characterize Heegaard splittings of totally orientable graph manifolds. Once we determine the possible intersections, we will be able to show that $k$ is a $(1,1)$ knot in $M$. Keep in mind that $k$ is not a torus knot, by Remark~\ref{remark:seifert_toroidal}. Therefore, $E(k)$ is not a Seifert manifold.

Consider a neighborhood $T \times I$ of $T$ in $E(k)$. This splits-up $E(k)$ into a totally orientable graph manifold $E(k)=Z' \cup (T \times I) \cup Z''$ where $Z'$ is a Seifert-fibered space over the disk with two exceptional fibers of orders $q_1>1$ and $q_2>1$, and $Z''$ is a Seifert-fibered space over the annulus with one exceptional fiber of order $q>1$. We consider $\d Z'$ as $T \times \{0\}$ and $\d Z'' - \d E(k)$ as $T \times \{1\}$. 

To put this in the context of \cite{schultens:heegaard}, we note that the \textit{vertex manifolds} are $Z'$ and $Z''$, and the \textit{edge manifold} is $T \times I$. The proof of \cite[Proposition 1]{thompson:disjoint} can be extended to manifolds with boundary, so we assume that $F$ is a strongly irreducible Heegaard surface of $E(k)$. By \cite[Lemma 6.1]{schultens:heegaard} there is a vertex or edge manifold $N$ of $E(k)$ and an isotopy of $F$ such that $F \cap \tilde{N}$ is incompressible for any vertex or edge manifold $\tilde{N} \ne N$. In this case, $N$ is called the \textbf{active component} of $F$. Throughout, denote $F \cap Z'$ by $F'$, $F \cap (T \times I)$ by $F_e$ and $F \cap Z''$ by $F''$.  

\begin{claim} \label {claim:active_edge}
We can isotope $T$ in $E(k)$ so that $T \times I$ is the active component.
\end{claim}  

\noindent \textit{Proof of claim.}
We proceed by contradiction. Suppose that the claim is false. By \cite[Proposition 7.5, Lemma 7.10]{schultens:heegaard}, $T$ may be isotoped in $E(k)$ so that 
\begin{itemize}
\item $F'$ is pseudo-horizontal, 
\item $F_e$ consists of incompressible annuli (with at least two spanning annuli), and
\item $F''$ consists of vertical annuli.
\end{itemize}
Since $F$ has genus $2$, $F'$ has a twice punctured torus component and possibly annulus components. 
Let $f \subset Z'$ be a fiber so that $\hat{F}=\cl{F' - N(f)}$ is horizontal in $\hat{Z}=\cl{Z' - N(f)}$.
By \cite[VI.25,VI.26,VI.28]{jaco:lectures} and \cite[Corollary 3.2]{waldhausen:irreducible}, the components of $\hat{F}$ are all parallel in $\hat{Z}$.
It follows that $F'$ is connected and that $f$ is nonseparating on $F'$; so $\hat{F}$ is a $4$--punctured sphere.

There are two cases:
\begin{itemize}
\item[(i)] $f$ is an exceptional fiber, or
\item[(ii)] $f$ is a regular fiber.
\end{itemize}
Suppose that we are in case (i). 
Thus $\hat{Z}$ is a cable space. By the classification in \cite[Lemma 3.1]{gl:incompressible}, a connected horizontal planar surface in $\hat{Z}$ must meet one of the boundary tori of $\hat{Z}$ in a single curve. By construction, $\hat{F}$ has a pair of boundary circles on each boundary torus of $\hat{Z}$, a contradiction. So case (i) cannot occur. 

Suppose that we are in case (ii). 
Then $\hat{Z}$ is a Seifert-fibered space over an annulus $\hat{B}$ with two exceptional fibers of orders $q_1$ and $q_2$. 
The map $\hat{Z} \rightarrow \hat{B}$ that projects fibers to points restricts to an $n$--fold branched covering $\hat{F} \rightarrow \hat{B}$. As explained in \cite{hatcher:notes}, the Euler characteristics of $\hat{F}$ and $\hat{B}$ are related by the formula
$$\chi(\hat{F}) = n \left( \chi(\hat{B}) \ - \ \sum_{i=1}^{2} \left(1 - \frac{1}{q_i}\right) \right) \ .$$
The equation simplifies to  
$-2 = n \left( -2 + \frac{1}{q_1} + \frac{1}{q_2} \right)$.
The only solution is $n=q_1=q_2=2$, but this makes $Z'$ the twisted $I$--bundle over the Klein bottle, a contradiction to the hypothesis of Proposition~\ref{prop:annular}. So case (ii) cannot occur. This completes the proof of the claim.
\vskip7pt
By Claim \ref{claim:active_edge}, we may isotope $T$ so that $T \times I$ is the active component. Since $F$ is a genus 2 surface, $F'$ and $F''$ must be vertical annuli. 

If part (1) of \cite[Theorem 1.1]{schultens:heegaard}  holds, then $F_e$ must contain a product annulus. This makes $\d F'$ parallel in $T \times I$ to $\d F''$ . This makes $E(k)$ a Seifert manifold, contrary to assumption.

Now suppose that part (2) of \cite[Theorem 1.1]{schultens:heegaard} holds. Then either $V \cap (T \times I)$ or $W \cap (T \times I)$ is a collar of $\Gamma=(c' \times 0) \cup (p \times I) \cup (c'' \times 1)$ where $c'$ and $c''$ are simple closed curves on $T$ with $c' \cap c'' = p$. See Figure~\ref{fig:active_component} and Figure~\ref{fig:active_component_gamma}. 

\begin{figure}[htbp]
  \centering
  
   \labellist
     \small \hair 2pt
     \pinlabel {$\color{green}{\bs{c' \times 0}}$} [br] at 118 501
     \pinlabel {$\color{red}{\bs{c'' \times 0}}$} [tr] at 67 241
     \pinlabel {$\color{green}{\bs{c' \times 1}}$} [bl] at 734 501
     \pinlabel {$\color{red}{\bs{c'' \times 1}}$} [tl] at 774 241
     \pinlabel {$\color{blue}{\bs{p \times I}}$} at 419 342
     \endlabellist  
  
  \includegraphics[width=3in]{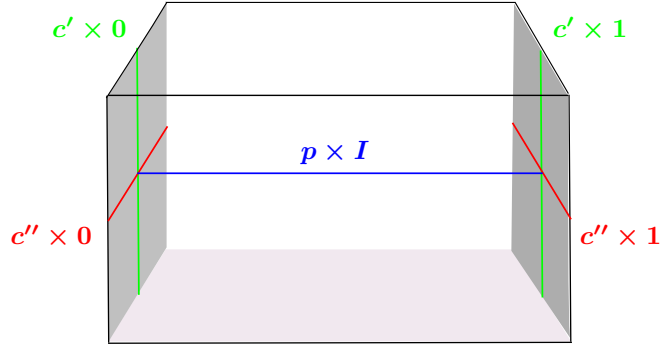}
  \caption{The active component $T \times I$ with indicated copies of $c'$ and $c''$ on $T \times 0$ and $T \times 1$.}
   \label{fig:active_component}
\end{figure}

\begin{figure}[htbp]
  \centering
  
   \labellist
     \small \hair 2pt
     \pinlabel {$\color{blue}{\Gamma}$} at 419 342
     \endlabellist  
  
  \includegraphics[width=3in]{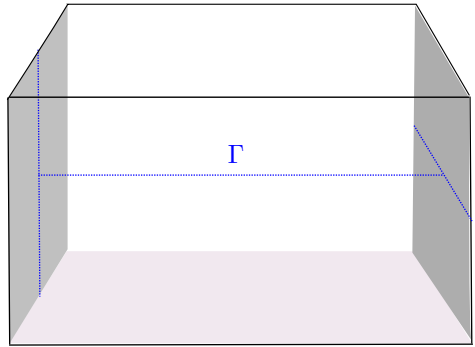}
  \caption{The graph $\Gamma$ is formed by the curves $c' \times 0$, $c'' \times 1$ and $p \times I$.}
  \label{fig:active_component_gamma}
\end{figure}

In the case that $V \cap (T \times I)$ is a neighborhood of $\Gamma$, a tunnel of the splitting can be obtained by taking a straight arc from $\d V$ through $T$ to a core of $Z'$. The case in which $W \cap (T \times I)$ is a neighborhood of $\Gamma$ is similar to the previous case, since $F_e$ cuts $T \times I$ into homeomorphic components.  It now follows from Lemma \ref{lem:1-bridge} that $k$ is a $(1,1)$ knot in $M$.
\end{proof}

\section{Proof of the Main Theorem}
\begin{thm*}
Let $k$ be a nonsimple prime tunnel number one knot in a lens space $M$, where $M$ contains no embedded Klein bottles. Then $k$ is a $(1,1)$ knot.
\end{thm*}

\begin{proof}
Assume the hypotheses. By Corollary~\ref{cor:haken}, $k$ is either a trivial knot, a torus knot, or a satellite knot. The theorem obviously holds if $k$ is a trivial knot or a torus knot. Now assume that $k$ is a satellite knot. By Propositions \ref{prop:anannular} and \ref{prop:annular}, we conclude that $k$ is a $(1,1)$ knot.  
\end{proof}

\bibliographystyle{plain}
\bibliography{tk1_Refs}

\end{document}